\documentclass[a4paper,11pt]{article}
\usepackage{amsfonts}
\usepackage{amsmath}
\usepackage{amssymb}
\usepackage{epsfig}
\usepackage{thmdefs}
\usepackage[round,authoryear]{natbib}

\numberwithin{equation}{section}
\newenvironment{proof}{\vspace{1ex}\noindent{\bf Proof:} }{
\mbox{}\hspace{\fill}\rule{1ex}{1ex}\vspace{1.5ex}}
\input{tcilatex}

\begin{document}

\title{Self-Similarity and Lamperti Convergence for \\
Families of Stochastic Processes}
\author{Bent J\o rgensen $\cdot $ J. Ra\'{u}l Mart\'{\i}nez $\cdot $ Clarice
G.B. Dem\'{e}trio \\
University of Southern Denmark\\
Universidad Nacional de C\'{o}rdoba, Argentina\\
University of S\~{a}o Paulo, Brazil}
\maketitle

\begin{abstract}
We define a new type of self-similarity for one-parameter families of
stochastic processes, which applies to a number of important families of
processes that are not self-similar in the conventional sense. This includes
a new class of fractional Hougaard motions defined as moving averages of
Hougaard L\'{e}vy process, as well as some well-known families of Hougaard L%
\'{e}vy processes such as the Poisson processes, Brownian motions with
drift, and the inverse Gaussian processes. Such families have many
properties in common with ordinary self-similar processes, including the
form of their covariance functions, and the fact that they appear as limits
in a Lamperti-type limit theorem for families of stochastic processes.
\end{abstract}

\section{Introduction\label{secone}}

The purpose of this paper is to extend the notion of self-similarity to
one-parameter families of stochastic processes. Recall that a real-valued
stochastic process $\left\{ X(t),t\geq 0\right\} $ is called \emph{%
self-similar with} \emph{Hurst exponent} $H\in (0,1)$ if it satisfies the
scaling property 
\begin{equation}
X(ct)\overset{d}{=}c^{H}X(t)\text{ for }t\geq 0\text{,}  \label{stoch}
\end{equation}%
for all $c>0$, where $\overset{d}{=}$ denotes equality of the
finite-dimensional distributions 
\citep{Embrechts2002a}%
. The corresponding process with drift $\mu \in \mathbb{R}$, defined by%
\begin{equation}
X(\mu ;t)=X(t)+\mu t\text{ for }t\geq 0\text{,}  \label{drift}
\end{equation}%
does not, however, satisfy (\ref{stoch}) for $\mu \neq 0,$ an example being
Brownian motion with drift. Instead, the family of processes (\ref{drift})
satisfies the following scaling property,%
\begin{equation}
X(\mu c^{H-1};ct)\overset{d}{=}c^{H}X(\mu ;t)\text{ for }t\geq 0\text{,}
\label{sdrift}
\end{equation}%
for all $c>0$, and we shall hence propose (\ref{sdrift}) as a new definition
of self-similarity for families of stochastic processes, cf. Definition \ref%
{ss-def} below.

We show that there are in fact many important families of stochastic
processes that satisfy (\ref{sdrift}) without having the drift form (\ref%
{drift}), and for values of $H$ that do not necessarily belong to $(0,1)$.
One such example is the class of Hougaard L\'{e}vy processes 
\citep{Lee1993}%
, which includes for example the family of Poisson processes ($H=0$),
certain gamma compound Poisson processes ($H<0$), and the family of inverse
Gaussian processes ($H=2$), cf. 
\citet{Wasan1968}%
. A further example is a new class of fractional Hougaard motions defined as
moving averages of Hougaard L\'{e}vy processes, generalizing fractional
Brownian motion. As we shall see, such processes have many properties in
common with ordinary self-similar processes, as reflected for example in the
familiar form of their covariance functions. This represents an important
step forward compared with conventional self-similar processes, where the
only processes with finite variance are the fractional Brownian motions.

In Section~\ref{secself-similar}, we present the new notion of
self-similarity for families of stochastic processes. We show that their
covariance structure is completely determined by the so-called variance
function, and we study the role of power variance functions. In Section~\ref%
{secExpo} we explore the connection between self-similarity and exponential
tilting, and show the self-similarity of Hougaard L\'{e}vy processes. In
Section~\ref{fracsystems} we show the self-similarity of the class of
fractional Hougaard motions using their moving average representation. In
Section~\ref{secLampertiTrans}, we introduce a Lamperti-type transformation,
which transforms a self-similar family into a family of stationary
processes. In Section~\ref{secLamperti} we consider some Lamperti-type limit
theorems, where families of self-similar processes appear as limits of
suitably scaled families of stochastic processes. Finally, in Section \ref%
{twinfty}, we investigate the case $H=1$ and its relation with exponential
variance functions.

\section{Self-similarity\label{secself-similar}}

\subsection{Definition\label{sec:def}}

We consider a family of real-valued stochastic processes $X=\left\{ X(\mu
;t):\mu \in \overline{\Omega },t\geq 0\right\} ,$ all defined on the same
probability space, where $t$ denotes time, and where the family is indexed
by a real parameter $\mu \in \overline{\Omega },$ where $\overline{\Omega }$
is an interval satisfying $\mathbb{R}_{+}\subseteq \overline{\Omega }$ and $c%
\overline{\Omega }\subseteq \overline{\Omega }$ for all $c>0$. We allow $%
\overline{\Omega }$ to contain the values $\pm \infty $, and define $\Omega =%
\mathrm{int}\overline{\Omega }$, the interior of $\overline{\Omega }$. These
conventions are important for the examples that we will discuss. Motivated
by the above discussion of processes with drift, we now propose an extended
definition of self-similarity for families of stochastic processes. In the
following we use the shorthand notation $X(\mu ;\cdot )=\left\{ X(\mu
;t):t\geq 0\right\} $ etc. to denote a particular stochastic process in $X$.

\begin{definition}
\label{ss-def}\textbf{(i)} A family of stochastic processes $X$ is called 
\emph{self-similar with Hurst exponent} $H\in \mathbb{R}$ and \emph{rate
parameter} $\mu \in \overline{\Omega }$ ($H$-SS) if 
\begin{equation}
X(\mu c^{H-1};ct)\overset{d}{=}c^{H}X(\mu ;t)\text{ for }t\geq 0\text{,}
\label{H-ss}
\end{equation}%
for all $c>0$ and $\mu \in \overline{\Omega }$. \textbf{(ii)} We say that $X$
has \emph{stationary increments} if each process in the family has
stationary increments, that is, for all $s>0$ and $\mu \in \overline{\Omega }
$ 
\begin{equation}
X(\mu ;s+t)-X(\mu ;s)\overset{d}{=}X(\mu ;t)-X(\mu ;0)\text{ for }t\geq 0%
\text{.}  \label{additive}
\end{equation}%
The family $X$ is called $H$-SSSI if it is $H$-SS and has stationary
increments.
\end{definition}

When necessary, we refer to the definition (\ref{H-ss}) as \emph{general
self-similarity}, whereas the conventional definition (\ref{stoch}) for
stochastic processes will be called \emph{strict self-similarity} from now
on; a terminology borrowed from stable distributions. The motivation for
expanding the use of the term self-similar in this way is that many
properties of strictly self-similar processes have immediate analogues for
general self-similar families, in particular for the covariance structure
(cf. Section~\ref{secCov}).

A family of processes with drift $X=\left\{ X(t)+\mu t:\mu \in \mathbb{R}%
,t\geq 0\right\} $ is general sense $H$-SS if and only if the stochastic
process $X(\cdot )$ is strictly self-similar with Hurst exponent $H$. It is
clear that general self-similarity constitutes an extension of strict
self-similarity, since if $X$ satisfies (\ref{H-ss}), then the distribution
of $X(\mu ;t)-\mu t$ will generally depend on $\mu $, as illustrated by the
examples in Sections \ref{ExHou} and \ref{fracsystems}. Nevertheless $\mu $
plays the role of rate for the process $X(\mu ;\cdot )$, in the sense that
the dimension of $\mu $ is the unit of the process $X(\mu ;\cdot )$ per unit
of time, thereby providing an extension of the idea of a drift parameter.
When applicable, the values $\mu =0$ and $\mu =\pm \infty $ in (\ref{H-ss})
correspond to the strictly self-similar processes $X(0;\cdot )$ and $X(\pm
\infty ;\cdot )$, respectively.

An equivalent way of writing the definition (\ref{H-ss}) is as follows,

\begin{equation}
c^{-H}X(\mu c^{H-1};ct)\overset{d}{=}X(\mu ;t)\text{ for }t\geq 0\text{,}
\label{fixedpoint}
\end{equation}%
for all $c>0$ and $\mu \in \overline{\Omega }$. Noting that the right-hand
side of (\ref{fixedpoint}) is independent of $c$, we observe a collapse of
the scaled marginal distributions of $X$ onto a single distribution
independent of $c$, much like the notion of universality in turbulence, see 
\cite{Barndorff-Nielsen2008} and references therein. Taking $c=1/t$ in (\ref%
{fixedpoint}), we obtain the useful relation%
\begin{equation}
t^{H}X(\mu t^{1-H};1)\overset{d}{\sim }X(\mu ;t)\text{ for }t>0\text{,}
\label{solve}
\end{equation}%
for each $\mu \in \overline{\Omega }$, where $\overset{d}{\sim }$ denotes
equality of the marginal distributions for each $t$. Given $H\in \mathbb{R}$%
, the left-hand side of (\ref{fixedpoint}) represents a multiplicative
transformation group $\left\{ R_{c}^{H}:c>0\right\} ,$ defined by 
\begin{equation}
R_{c}^{H}X(\mu ;t)=c^{-H}X(\mu c^{H-1};ct)\text{.}  \label{group}
\end{equation}%
This is a continuous analogue of the \emph{renormalization group} in the
sense of \cite{Jona-Lasinio1975}, see also 
\citet[p.~15]{ Embrechts2002a}%
, with fixed point given by (\ref{fixedpoint}), a topic that we shall return
to in Section~\ref{secLamperti}.

The general interpretation of the definition (\ref{H-ss}) is that a
rescaling of $X(\mu ;t)$ is equivalent in distribution to a simultaneous
rescaling of time $t$ and rate $\mu ,$ so that non-trivial solutions to (\ref%
{H-ss}) require that the different processes of the family are suitably
linked, for example by a drift term, as in (\ref{drift}), or by means of
exponential tilting, cf. Section~\ref{secExpo}. In fact, (\ref{solve})
suggests that a simple rescaling $t^{H}X(\mu t^{1-H})$ of a given process $%
X(\cdot )$ provides a trivial solution to (\ref{H-ss}). Similarly, for $%
H\neq 1$ and $\mu \neq 0$ we may take $c=\left\vert \mu \right\vert
^{1/(1-H)}$ in (\ref{fixedpoint}), which yields 
\begin{equation*}
X(\mu ;t)\overset{d}{=}\left\vert \mu \right\vert ^{H/(H-1)}X(\mathrm{sgn}%
(\mu );\left\vert \mu \right\vert ^{1/(1-H)}t)\text{ for }t\geq 0\text{,}
\end{equation*}%
where $\mathrm{sgn}(\mu )$ is the sign of $\mu $. Hence, if $X(\cdot )$ is a
given stochastic process, then the rescaled family defined by $X(\mu ;t)=\mu
^{H/(H-1)}X(\mu ^{1/(1-H)}t)$ for $\mu ,t>0$ is also a trivial solution to (%
\ref{H-ss}).

Note that taking $t=0$ in (\ref{H-ss}) gives $X(\mu c^{H-1};0)\overset{d}{%
\sim }c^{H}X(\mu ;0)$, which shows that if $X(\mu ;0)\overset{d}{\sim }0$
for one value of $\mu $, then the same is the case for all values of $\mu $
with the same sign. In most cases we have $X(\mu ;0)\equiv 0$.

\subsection{Covariance structure\label{secCov}}

Let $X$ be an $H$-SSSI family of stochastic processes. If $X(\mu ;1)$ has
finite second moments for all $\mu \in \Omega $ (second-moment assumptions)
then (\ref{solve}) implies that $X(\mu ;t)$ has finite second moments for
all $t>0$ and $\mu \in \Omega $. We shall now explore the covariance
structure for such families. Most of the results in the following require $%
H\neq 1,$ and the case $H=1$ will be considered separately in Section \ref%
{twinfty}. The next result concerns the structure of the first moment.

\begin{proposition}
\label{propmean}Let $X$ be an $H$-SSSI family of stochastic processes with $%
H\neq 1$, and assume that $X(\mu ;t)$ has finite expectation for all $\mu
\in \Omega $ and $t\geq 0$. If the function $t\mapsto \mathbb{E}\left[ X(\mu
;t)\right] $ is continuous on $[0,\infty )$ for each $\mu \in \Omega $, then
there exist functions $a$ and $b$ such that 
\begin{equation}
\mathbb{E}\left[ X(\mu ;t)\right] =a(\mathrm{sgn}(\mu ))\left\vert \mu
\right\vert ^{H/(H-1)}+b(\mathrm{sgn}(\mu ))\left\vert \mu \right\vert t%
\text{ for }t\geq 0\text{ and }\mu \in \Omega \text{.}  \label{rate}
\end{equation}%
If $\Omega =\mathbb{R}$ and $0\leq H<1$, then continuity of the function $%
\mu \mapsto \mathbb{E}\left[ X(\mu ;t)\right] $ implies $a\equiv 0$.
\end{proposition}

\begin{proof}
For given $\mu \in \Omega $, the assumption of stationary increments (\ref%
{additive}) implies that, for $s,t\geq 0$, 
\begin{equation*}
\mathbb{E}\left[ X(\mu ;s+t)\right] +\mathbb{E}\left[ X(\mu ;0)\right] =%
\mathbb{E}\left[ X(\mu ;s)\right] +\mathbb{E}\left[ X(\mu ;t)\right] \text{.}
\end{equation*}%
Using standard results for the Cauchy functional equation 
\citep[cf.][p.~4]{Bingham1987}%
, the continuity of $\mathbb{E}\left[ X(\mu ;\cdot )\right] $ implies that 
\begin{equation}
\mathbb{E}\left[ X(\mu ;t)\right] =a(\mu )+b(\mu )t\text{ for }t\geq 0\text{,%
}  \label{themean}
\end{equation}%
for suitable function $a$ and $b$ defined on $\Omega $. In view of (\ref%
{H-ss}), we obtain, for all $c,t>0$ and $\mu \in \Omega $, 
\begin{equation*}
a(\mu c^{H-1})+b(\mu c^{H-1})ct=c^{H}a(\mu )+c^{H}b(\mu )t\text{,}
\end{equation*}%
which implies that 
\begin{equation*}
a(\mu c^{H-1})=c^{H}a(\mu )\text{ and }b(\mu c^{H-1})=c^{H-1}b(\mu )\text{.}
\end{equation*}%
By inserting $\mu =\pm 1$ and $x=\pm c^{H-1}$ in these equations, we obtain $%
a(x)=\left\vert x\right\vert ^{H/(H-1)}a(\mathrm{sgn}(x))$ and $%
b(x)=\left\vert x\right\vert b(\mathrm{sgn}(x))$, in agreement with (\ref%
{rate}) for $\mu \neq 0$.

If $\Omega =\mathbb{R}$, such that $0\in \Omega $, then we have already seen
in Section~\ref{sec:def} that the process $X(0;\cdot )$ is strictly
self-similar, in which case it is well-known that $X(0;0)\equiv 0$. This is
in agreement with the limit of (\ref{rate}) as $\mu \rightarrow 0$, provided
that either $a\equiv 0$ or $\left\vert \mu \right\vert ^{H/(H-1)}\rightarrow
0$ as $\mu \rightarrow 0$, the latter being the case outside the interval $%
0\leq H\leq 1$. Hence in the case $0\leq H<1$, continuity of the function $%
\mu \mapsto \mathbb{E}\left[ X(\mu ;t)\right] $ implies $a\equiv 0$. This
completes the proof.
\end{proof}

The constant term $a(\mathrm{sgn}(\mu ))\left\vert \mu \right\vert
^{H/(H-1)} $ may be removed from the processes by subtraction. Thus, if the
family $X$ satisfies the conditions of Proposition \ref{propmean}, we define
the corresponding \emph{centered family} $X_{0}$ by%
\begin{equation}
X_{0}(\mu ;t)=X(\mu ;t)-a(\mathrm{sgn}(\mu ))\left\vert \mu \right\vert
^{H/(H-1)}\text{.}  \label{famcen}
\end{equation}%
It is easy to check that the family $X_{0}$ is again $H$-SSSI, now with $a(%
\mathrm{sgn}(\mu ))\equiv 0$. In many cases the mean function (\ref{rate})
for the centered family $X_{0}$ takes the following simple form 
\begin{equation}
\mathbb{E}\left[ X_{0}(\mu ;t)\right] =\mu t\text{ for }t\geq 0\text{ and }%
\mu \in \Omega \text{,}  \label{centeredmean}
\end{equation}%
corresponding to $b(\mu )=\mu $. The latter form may be assumed without loss
of generality (i.e. up to a rescaling of $\mu $) in the important case $%
\Omega =\mathbb{R}_{+}$ whereas for $\Omega =\mathbb{R}$ this requires $%
b(1)=b(-1)$. From now on we assume that the family $X$ is centered, unless
otherwise stated.

We shall now derive the covariance structure of the family $X$ under
second-moment assumptions, which is done in much the same way as for
strictly self-similar processes \citep[cf.][]{Taqqu2003}. Let us introduce
the \emph{variance function} $V:\Omega \rightarrow \mathbb{R}_{+}$ defined
by\ 
\begin{equation}
V(\mu )=\mathrm{Var}\left[ X(\mu ;1)\right] \text{ for }\mu \in \Omega .
\label{unitvf}
\end{equation}%
Consider a given value of $\mu \in \Omega $. The marginal variance of the
process may be obtained by calculating the variance on both sides of (\ref%
{solve}), giving 
\begin{equation}
\mathrm{Var}\left[ X(\mu ;t)\right] =t^{2H}V(\mu t^{1-H})=V_{H}(\mu ;t)\text{%
,}  \label{variance}
\end{equation}%
say, for all $t>0$. The whole covariance structure of $X$ may now be
expressed in terms of the variance function $V$. Thus, for $s,t>0$ we obtain 
\begin{equation}
\mathrm{Var}\left[ X(\mu ;t)-X(\mu ;s)\right] =\mathrm{Var}\left[ X(\mu ;s)%
\right] +\mathrm{Var}\left[ X(\mu ;t)\right] -2\mathrm{Cov}\left[ X(\mu
;s),X(\mu ;t)\right] \text{.}  \label{Var}
\end{equation}%
Using the stationarity of the increments together with (\ref{variance}),
this gives the following covariance function for the process $X(\mu ;\cdot )$
for $s,t>0$, 
\begin{equation}
\mathrm{Cov}\left[ X(\mu ;s),X(\mu ;t)\right] =\frac{1}{2}\left[ V_{H}\left(
\mu ;s\right) +V_{H}\left( \mu ;t\right) -V_{H}\left( \mu ;\left\vert
t-s\right\vert \right) \right] \text{.}  \label{covariance}
\end{equation}

\subsection{Power variance functions\label{secpower}}

Let us consider the important case where a centered $H$-SSSI family of
stochastic processes $X$ has power variance function $V(\mu )=\sigma ^{2}\mu
^{p}$ for some $p\in \mathbb{R}$ and $\sigma ^{2}>0$, where $p$ is called
the \emph{power parameter}. Here we take the interior of the domain for $\mu 
$ to be $\Omega _{0}=\mathbb{R}$ for $p=0$ and $\Omega _{p}=\mathbb{R}_{+}$
for $p\neq 0$. In this case, the variance (\ref{variance}) becomes 
\begin{equation}
\mathrm{Var}\left[ X(\mu ;t)\right] =\sigma ^{2}\mu ^{p}t^{2H+p(1-H)}=\sigma
^{2}\mu ^{p}t^{2-D}\text{ for }t>0\text{,}  \label{p-var}
\end{equation}%
say, where $D$ is the \emph{fractal dimension} defined by 
\begin{equation}
D=\left( H-1\right) (p-2)\text{,}  \label{D-def}
\end{equation}%
with domain $D\in \lbrack 0,2]$ (cf. Lemma \ref{H-sign} below).

The covariance function (\ref{covariance}) now takes the form%
\begin{equation}
\mathrm{Cov}\left[ X(\mu ;s),X(\mu ;t)\right] =\sigma ^{2}\mu ^{p}R_{D}(s,t)%
\text{ for }s,t>0\text{,}  \label{Covar}
\end{equation}%
where 
\begin{equation*}
R_{D}(s,t)=\frac{1}{2}\left[ s^{2-D}+t^{2-D}-\left\vert t-s\right\vert ^{2-D}%
\right] \text{ for }s,t>0\text{.}
\end{equation*}%
Compared with (\ref{covariance}), the covariance function (\ref{Covar}) is
now a product of the variance function $\sigma ^{2}\mu ^{p}$ and a function
depending on $s$ and $t$ only$.$ Let us also consider the correlation
between two non-overlapping increments, which may be expressed in terms of $%
r=\sqrt{s/t}$,%
\begin{equation}
\mathrm{Corr}\left[ X(\mu ;s),X(\mu ;s+t)-X(\mu ;s)\right] =\frac{1}{2}\left[
\left( r^{-1}+r\right) ^{2-D}-r^{2-D}-r^{-2+D}\right] \text{ for }s,t>0\text{%
.}  \label{CovCorr}
\end{equation}

\begin{lemma}
\label{H-sign}The function $R_{D}$ is non--negative definite if and only if $%
0\leq D\leq 2.$
\end{lemma}

\begin{proof}
Let us first take $p=0$ in (\ref{D-def}), corresponding to $D=2\left(
1-H\right) ,$ in which case 
\begin{equation*}
R_{2\left( 1-H\right) }(s,t)=\frac{1}{2}\left[ s^{2H}+t^{2H}-\left\vert
t-s\right\vert ^{2H}\right]
\end{equation*}%
has the same functional form as the covariance function of standard
fractional Brownian motion, which is known to be non-negative definite for $%
0\leq H\leq 1$ 
\citep{Taqqu2003,Embrechts2002a}%
. It follows that $R_{D}$ is non-negative definite for any $D\in \lbrack
0,2] $. Let us now show the necessity of this condition. Since the
correlation (\ref{CovCorr}) must be less than or equal to 1, we obtain for $%
r=1$ that $\left( 2^{2-D}-1-1\right) /2\leq 1$, which is equivalent to $%
D\geq 0$. Now consider the case $D>2,$ where the function on the right-hand
side of (\ref{CovCorr}) behaves asymptotically as $-r^{-2+D}/2$ as $%
r\rightarrow \infty $. For $r$ large this is in contradiction with (\ref%
{CovCorr}) being a correlation, so that we must have $D\leq 2$.
\end{proof}

\begin{lemma}
\label{D-sign}Let $X$ be an $H$-SSSI family with power variance function.
Then

\begin{enumerate}
\item \label{Four}The correlation (\ref{CovCorr}) is positive for $0\leq D<1$
and negative for $1<D\leq 2$.

\item \label{Three}The process $X(\mu ;\cdot )$ has uncorrelated increments
if and only if $D=1$, where (\ref{Covar}) becomes 
\begin{equation}
\mathrm{Cov}\left[ X(\mu ;s),X(\mu ;t)\right] =\frac{1}{2}\sigma ^{2}\mu
^{p}\min \left\{ s,t\right\} \text{.}  \label{minst}
\end{equation}
\end{enumerate}
\end{lemma}

\begin{proof}
Item \ref{Four}. The correlation (\ref{CovCorr}) is positive if and only if%
\begin{equation*}
\left( r^{-1}+r\right) ^{2-D}>r^{2-D}+r^{-2+D}\text{.}
\end{equation*}%
In terms of the strictly convex function $f(x)=\log \left(
e^{x}+e^{-x}\right) ,$ this is equivalent to the inequality 
\begin{equation}
(2-D)f(x)>f\left( x(2-D)\right) \text{.}  \label{DDD}
\end{equation}%
Using the convexity of $f$ along with the fact that $f(x)$ behaves
asymptotically as $\left\vert x\right\vert $ for $\left\vert x\right\vert $
large, we find that the inequality (\ref{DDD}) is satisfied for $0\leq D<1$,
whereas the opposite inequality holds for $1<D\leq 2,$ which proves Item \ref%
{Four}.

Item \ref{Three}. The correlation (\ref{CovCorr}) is zero if $D=1$, which by
Item \ref{Four} is also necessary.
\end{proof}

It is convenient at this point to introduce the parameter $\alpha \in
\lbrack -\infty ,\infty )$, defined by the following one-to-one
transformation of $p\in (-\infty ,\infty ]$, 
\begin{equation}
\alpha =\alpha (p)=1+(1-p)^{-1}\text{,}  \label{alpha}
\end{equation}%
with the conventions that $\alpha (1)=-\infty $ and $\alpha (\infty )=1$. By
Item \ref{Three} of Lemma \ref{D-sign}, we find that the correlation between
two non-overlapping increments (\ref{CovCorr}) is zero if and only if $%
H=1/\alpha $, with the convention that $\alpha =-\infty $ corresponds to $%
H=0 $. A further reason for our interest in the parameter $\alpha $ comes
from the connection with $\alpha $-stable distributions in the case $\alpha
\in (0,2]$, see Section \ref{ExHou}, where we discuss a class of
self-similar families of L\'{e}vy processes.

\begin{table}[tbp] \centering%
\begin{tabular}{|c|c|c|}
\hline
$p$ & Positive correlation & Negative correlation \\ \hline
$p<2$ ($\alpha \notin \lbrack 0,1]$) & $1/\alpha <H\leq 1$ & $\left(
2-\alpha \right) /\alpha \leq H<1/\alpha $ \\ 
$p>2$ ($\alpha \in (0,1)$) & $1\leq H<1/\alpha $ & $1/\alpha <H\leq \left(
2-\alpha \right) /\alpha $ \\ \hline
\end{tabular}%
\caption{the sign of the correlation (\protect\ref{CovCorr}) as a function of
$\alpha$ and $H$.}\label{TableKey}%
\end{table}%
We may now express the domain for the Hurst exponent $H,$ for given $\alpha $%
, as an interval with endpoints $1$ and $\left( 2-\alpha \right) /\alpha $,
which follows from the domain $0\leq D\leq 2$ via (\ref{D-def}) and (\ref%
{alpha}). This domain is summarized in Table \ref{TableKey}, along with the
sign of the correlation (\ref{CovCorr}) (Item \ref{Four} of Lemma \ref%
{H-sign}). By comparison, in the case of strict self-similarity, the Hurst
parameter is restricted to the domain $H>0$, and to the smaller domain $%
0<H\leq 1$ under second-moment assumptions, cf. 
\citet{Taqqu2003}%
.

The importance of power variance functions will become clear in Sections \ref%
{ExHou} and \ref{fracsystems}, where we give examples of processes with
independent increments and with correlated increments, respectively, giving
rise to $H$-SSSI families corresponding to many of the possible combinations
of $H$ and $p$ discussed above. We note in this connection that $D=0$ and $%
H\neq 1$ implies $p=2$ $(\alpha =0)$, in which case the covariance function (%
\ref{Covar}) takes the form 
\begin{equation}
\mathrm{Cov}\left[ X(\mu ;s),X(\mu ;t)\right] =\sigma ^{2}\mu ^{2}st\text{
for }s,t>0\text{.}  \label{gamma-case}
\end{equation}%
We shall return to this case in Section \ref{ExHou}. The other case where $%
D=0$, namely $H=1$, will be considered in Section \ref{twinfty}.

\section{Exponential tilting and Hougaard L\'{e}vy processes\label{secExpo}}

\subsection{Exponential tilting}

We now consider the important case where $X$ is a natural exponential family
of centered stochastic processes, cf. \cite{Kuchler1997}. Let $P_{\mu }^{t}$
denote the marginal distribution of $X(\mu ;t)$ under the probability
measure $P_{\mu },$ say, and let $\mathbb{E}_{\mu }$ and $\mathbb{E}_{\mu
}^{t}$ denote expectation under $P_{\mu }$ and $P_{\mu }^{t}$, respectively.
Let $\kappa $ denote the cumulant generating function of $P_{1}^{1}$,
defined by 
\begin{equation}
\kappa (\theta )=\log \mathbb{E}_{1}\left[ e^{\theta X(1;1)}\right] ,
\label{kappa}
\end{equation}%
with effective domain $\Theta =\left\{ \theta \in \mathbb{R}:\kappa (\theta
)<\infty \right\} $, assumed to have non-empty interior. Assuming that $%
\boldsymbol{P}^{t}=\left\{ P_{\mu }^{t}:\mu \in \overline{\Omega }\right\} $
is a natural exponential family for $t=1,$ we shall see that this is then
the case for all $t>0$. We then call $X$ an \emph{exponential} $H$-SSSI 
\emph{family}.

Let us consider the exponential change of measure from $P_{1}$ to $P_{\mu }$%
, corresponding to the likelihood ratio 
\begin{equation}
K_{\mu }(t)=\exp \left[ \theta _{1}^{t}X(1;t)-\kappa _{1}^{t}(\theta
_{1}^{t})\right] \text{,}  \label{exponential}
\end{equation}%
where $\kappa _{\mu }^{t}$ is the cumulant generating function for $P_{\mu
}^{t}$ and the canonical parameter $\theta _{1}^{t}$ is chosen such that $%
\mathbb{E}_{1}\left[ X(1;t)\right] =t$. The exponentially tilted probability
measure $P_{\mu }$ may then be defined from $P_{1}$ as follows: 
\begin{equation}
P_{\mu }^{t}(A)=\mathbb{E}_{1}\left[ 1_{A}K_{\mu }(t)\right] \text{,}
\label{exponential1}
\end{equation}%
where $1_{A}$ denotes the indicator function for the Borel set $A$. More
generally we have 
\begin{equation}
\mathbb{E}_{\mu }^{t}\left[ U\right] =\mathbb{E}_{1}\left[ UK_{\mu }(t)%
\right] \text{,}  \label{TiltingMean}
\end{equation}%
where $U$ is a non-negative random variable.

We shall now derive the cumulant generating function $\kappa _{\mu }^{t}$
for general $\mu $ and $t=1$. Standard exponential family arguments along
with (\ref{TiltingMean}) yield 
\begin{eqnarray}
\kappa _{\mu }^{1}(u) &=&\log \mathbb{E}_{\mu }\left[ e^{uX(\mu ;1)}\right] 
\notag \\
&=&\log \mathbb{E}_{1}\left[ K_{\mu }(1)e^{uX(1;1)}\right]  \notag \\
&=&\log \mathbb{E}_{1}\left[ e^{\left( \theta _{1}^{1}+u\right)
X(1;1)-\kappa (\theta _{1}^{1})}\right]  \notag \\
&=&\kappa (u+\theta _{1}^{1})-\kappa (\theta _{1}^{1})\text{,}
\label{Laplace}
\end{eqnarray}%
with effective domain $\Theta -\theta _{1}^{1}$. The mean and variance
function of $X(\mu ;1)$ are $\mu =\dot{\kappa}(\theta _{1}^{1})$ with domain 
$\overline{\Omega }=\dot{\kappa}(\Theta )$, and 
\begin{equation}
V(\mu )=\ddot{\kappa}\circ \dot{\kappa}^{-1}(\mu )\text{ for }\mu \in \Omega 
\text{,}  \label{Vfu}
\end{equation}%
respectively, where dots denote derivatives. The assumption that $P_{1}^{1}$
has mean $1$ implies that $\dot{\kappa}(0)=1$. It is well known that the
variance function $V$ together with its domain $\Omega $ characterize the
exponential family (\ref{Laplace}), cf. \citet[Ch.~2]{Jorgensen1997}.

Turning now to the case of general $t>0$, we obtain by means of (\ref{solve}%
) the following expression for the cumulant generating function $\kappa
_{\mu }^{t}$, 
\begin{eqnarray}
\kappa _{\mu }^{t}(u) &=&\log \mathbb{E}_{\mu }\left[ e^{uX(\mu ;t)}\right] 
\notag \\
&=&\log \mathbb{E}_{\mu }\left[ \exp \left( ut^{H}X(\mu t^{1-H};1)\right) %
\right]  \notag \\
&=&\kappa (ut^{H}+\theta _{\mu }^{t})-\kappa (\theta _{\mu }^{t})\text{,} 
\notag \\
&=&\kappa \left( t^{H}\left( u+\theta _{\mu }^{t}/t^{H}\right) \right)
-\kappa \left( t^{H}\left( \theta _{\mu }^{t}/t^{H}\right) \right) \text{,}
\label{kappa2}
\end{eqnarray}%
for values of $u$ such that $ut^{H}+\theta _{\mu }^{t}\in \Theta $, where $%
\theta _{\mu }^{t}$ is the solution to $\dot{\kappa}(\theta _{\mu
}^{t})t^{H-1}=\mu $. In particular, this implies that $X(\mu ;t)$ has mean $%
\mu t$ and variance given by (\ref{variance}), in agreement with the results
of Section~\ref{secCov}.

The form (\ref{kappa2}) shows that $\boldsymbol{P}^{t}$ is indeed a natural
exponential family for each $t>0$, corresponding to the cumulant generator $%
\kappa (t^{H}\cdot )$ and canonical parameter $\theta _{\mu }^{t}/t^{H}$. We
may interpret (\ref{kappa2}) as representing a translation structure in the
Fourier domain, as compared with the drift term $\mu t$ in (\ref{drift}),
which represents a translation structure in the sample space.

\subsection{Self-similarity of Hougaard L\'{e}vy processes\label{ExHou}}

We now consider the case where $X$ is a centered $H$-SSSI family of L\'{e}vy
processes. Under second-moment assumptions, the independence of the
increments implies the following form for the covariance function: 
\begin{equation}
\mathrm{Cov}\left[ X(\mu ;t),X(\mu ;t+s)\right] =\mathrm{Var}\left[ X(\mu ;t)%
\right] =tV(\mu )\text{ for }s,t>0.  \label{Houvar}
\end{equation}%
Under the further assumption that the family $X$ is exponential, it is clear
from the results of Section~\ref{secExpo} that the variance function $V$
with its domain $\Omega $ characterize the family $X$ among all exponential $%
H$-SSSI families, because $V$ characterizes the exponential family of
marginal distributions of $X(\mu ;1)$, which, in turn, characterizes the
corresponding family of L\'{e}vy processes.

From now on, we denote the marginal distribution of $t^{-1}X(\mu ;t)$ by the
symbol $\mathrm{ED}(\mu ,t)$, where $\mathrm{ED}(\mu ,t)$ is an \emph{%
exponential dispersion model} in the sense of 
\citet[Ch.~3]{Jorgensen1997}%
. This distribution has mean $\mu $ and variance $t^{-1}V(\mu )$ for $t>0$
and $\mu \in \Omega $. The above characterization of the family $X$ in terms
of the variance function $V$ may now be rephrased as follows. Namely, the
variance function $V$ with domain $\Omega $ characterizes the exponential
dispersion model $\mathrm{ED}(\mu ,t)$ among all exponential dispersion
models, cf. 
\citet[Ch.~3]{Jorgensen1997}%
.

Following \cite{jorgensen1992} and \cite{Lee1993}, we define a \emph{%
Hougaard L\'{e}vy process}$\ S_{p}(\mu ;t)$ to be an exponential family of L%
\'{e}vy process with power variance function $V(\mu )=\sigma ^{2}\mu ^{p}$,
see also 
\citet{Hougaard1997}%
, 
\citet[Ch.~4]{Jorgensen1997}
and \cite{Vinogradov2008}. We define the \emph{Tweedie distribution} $%
\mathrm{Tw}_{p}(\mu ,t)$ to be the corresponding exponential dispersion
model, such that 
\begin{equation}
t^{-1}S_{p}(\mu ;t)\overset{d}{\sim }\mathrm{Tw}_{p}(\mu ,t)\text{,}
\label{Twee}
\end{equation}%
with mean $\mu $ and variance $t^{-1}\sigma ^{2}\mu ^{p}$ (note that we
suppress the parameter $\sigma ^{2}$ in the notation). The parameter domains
are $p\in \Delta =\mathbb{R\diagdown }(0,1)$, $\sigma ^{2}>0,$ and $\mu \in 
\overline{\Omega }_{p}$, where $\overline{\Omega }_{p}$ is defined by 
\begin{equation*}
\overline{\Omega }_{p}=\left\{ 
\begin{array}{lll}
\lbrack 0,\infty ) & \text{ for } & p<0 \\ 
\mathbb{R} & \text{ for } & p=0 \\ 
\mathbb{R}_{+} & \text{ for } & 1\leq p\leq 2 \\ 
(0,\infty ] & \text{ for } & p>2\text{.}%
\end{array}%
\right.
\end{equation*}

The Tweedie distribution satisfies the following scaling property for all $%
p\in \Delta $, $\mu \in \Omega _{p}$ and $t>0$, 
\begin{equation}
c\mathrm{Tw}_{p}(\mu ,t)=\mathrm{Tw}_{p}(c\mu ,c^{p-2}t)\text{ for }c>0\text{%
,}  \label{Tw}
\end{equation}%
which characterizes the Tweedie model among all exponential dispersion
models, see 
\citet[Ch.~4]{Jorgensen1997}%
. The Hougaard L\'{e}vy process $S_{p}(\mu ;t)$ is hence characterized by
the following distribution of the increments,%
\begin{equation}
S_{p}(\mu ;t)\overset{d}{\sim }t\mathrm{Tw}_{p}(\mu ,t)=\mathrm{Tw}_{p}(\mu
t,t^{p-1})\text{,}  \label{average}
\end{equation}%
with variance $t\sigma ^{2}\mu ^{p}$. We shall now characterize the Hougaard
L\'{e}vy processes in terms of self-similarity. Recall that the parameter $%
\alpha $ is defined as a function of $p$ by (\ref{alpha}), and note that $%
p\in \Delta $ corresponds to the domain $\alpha \in \lbrack -\infty ,1)\cup
(1,2]$.

\begin{theorem}
\label{Hougaard}The family of Hougaard L\'{e}vy processes $S_{p}(\mu ;\cdot
) $ may be characterized as follows. Let $X$ be an exponential family of L%
\'{e}vy processes. Then $X$ is self-similar with Hurst exponent%
\begin{equation}
H=\frac{1}{\alpha }\text{,}  \label{H-def}
\end{equation}%
with the convention that $\alpha =-\infty $ ($p=1$) corresponds to $H=0$, if
and only if $X$ has power variance function with power $p\neq 2$.
\end{theorem}

Note that the value $H=1/\alpha $ is in agreement with the condition for
uncorrelated increments of Lemma \ref{H-sign}, Item \ref{Three}.

\begin{proof}
Let us first show that the Hougaard L\'{e}vy process $S_{p}(\mu ;t)$ is $%
1/\alpha $-SS. Since s Hougaard L\'{e}vy process has stationary and
independent increments, it suffices to consider the marginal distribution of 
$S_{p}(\mu ;t)$, so that we must show%
\begin{equation*}
S_{p}(\mu c^{H-1};ct)\overset{d}{\sim }c^{H}S_{p}(\mu ;t)\text{ for }t>0
\end{equation*}%
with $H=1/\alpha $. Using (\ref{average}), the left-hand side has marginal
distribution%
\begin{equation*}
S_{p}(\mu c^{H-1};ct)\overset{d}{\sim }\mathrm{Tw}_{p}(\mu tc^{H},\left(
ct\right) ^{p-1})\text{.}
\end{equation*}%
The scaling property (\ref{Tw}) yields the following marginal distribution
for the right-hand side, 
\begin{eqnarray*}
c^{H}S_{p}(\mu ;t) &\overset{d}{\sim }&c^{H}\mathrm{Tw}_{p}(\mu t,t^{p-1}) \\
&=&\mathrm{Tw}_{p}(\mu tc^{H},c^{H\left( p-2\right) }t^{p-1}) \\
&=&\mathrm{Tw}_{p}\left( \mu tc^{H},\left( ct\right) ^{p-1}\right) \text{,}
\end{eqnarray*}%
where we have used the fact that $H\left( p-2\right) =p-1$, which follows
from (\ref{alpha}) and (\ref{H-def}). This implies that $S_{p}(\mu ;t)$ is $%
H $-SSSI with $H$ given by (\ref{H-def}).

To show the reverse implication, let us assume that the exponential family
of L\'{e}vy processes $X$ is self-similar with Hurst exponent $H=1/\alpha $.
By (\ref{Houvar}) the variance of the process is $\mathrm{Var}\left[ X(\mu
;t)\right] =tV(\mu )$, and comparing with (\ref{variance}), we obtain the
equation%
\begin{equation}
tV(\mu )=t^{2H}V(\mu t^{1-H})\text{.}  \label{power}
\end{equation}%
By taking $\mu =1$ in (\ref{power}) and redefining $\mu $ to be $\mu
=t^{1-H} $, we obtain the following solution:%
\begin{equation*}
V(\mu )=V(1)\mu ^{\left( 1-2H\right) /\left( 1-H\right) }\text{,}
\end{equation*}%
and using (\ref{H-def}) and (\ref{alpha}) we find the power to be $\left(
1-2H\right) /\left( 1-H\right) =p$. Hence $V$ is a power variance function
with power parameter $p$ and $\sigma ^{2}=V(1)$.
\end{proof}

\section{Fractional Hougaard motion\label{fracsystems}}

We shall now consider stochastic integration with respect to a L\'{e}vy
process, which in turn will be used to construct a class of fractional
Hougaard motions as moving averages of Hougaard L\'{e}vy processes. Some of
the results in the following are parallel to results of 
\citet{Marquardt2006}%
, but unlike Marquardt, who considers L\'{e}vy processes with zero mean,
finite variance, and no Brownian component, we shall consider stochastic
integration with respect to an arbitrary L\'{e}vy process.

\subsection{Stochastic integration with respect to a L\'{e}vy process}

We now consider stochastic integration with respect to a L\'{e}vy process $X$%
. We follow the approach of \cite{Barndorff-Nielsen2006}, which has the
advantage that we work directly with the cumulant transform (log
characteristic function) and the cumulants, although there are more general
approaches to stochastic integration with respect to a L\'{e}vy process
available, see e.g. \cite{Rajput1989}.

We shall define the stochastic integral%
\begin{equation}
\int_{a}^{b}f(u)\,dX(u),  \label{stint}
\end{equation}%
for a non-random real function $f$. Initially we consider the case where $%
0\leq a<b<\infty $, after which we extend the integration interval to the
whole real line. We denote the cumulant transform of an infinitely divisible
random variable $Y$ by 
\begin{equation*}
C(z;Y)=\log \mathbb{E}\left[ e^{izY}\right] \text{.}
\end{equation*}%
Due to the infinite divisibility, the characteristic function $\mathbb{E}%
\left[ e^{izY}\right] $ has no zeroes, and hence $C$ is well-defined by
means of the principal branch of the complex logarithm. Recall that 
\begin{equation}
\mathbb{E}\left[ Y\right] =-i\dot{C}(0;Y)\text{ and }\mathrm{Var}\left[ Y%
\right] =-\ddot{C}(0;Y)\text{,}  \label{meanvar}
\end{equation}%
provided that $C$ is twice differentiable at zero, where $\dot{C}(z;Y)$ and $%
\ddot{C}(z;Y)$ denote the first and second derivatives of $C$, respectively,
with respect to $z$.

\begin{lemma}
\label{Lemma}Assume that $0\leq a<b<\infty $. Let $X$ be a L\'{e}vy process,
and let the function $f:[a,b]\rightarrow \mathbb{R}$ be continuous. Then the
stochastic integral (\ref{stint}) exists in the limit, in probability, of
approximating Riemann sums. The distribution of the random variable (\ref%
{stint}) is infinitely divisible, and has cumulant transform given by 
\begin{equation}
C\left( z;\int_{a}^{b}f(u)\,dX(u)\right) =\int_{a}^{b}C\left(
zf(u);X(1)\right) \,du\text{ for }z\in \mathbb{R}\text{.}  \label{cumcum}
\end{equation}%
If $X(t)$ has finite second moments, the mean and variance of the random
variable (\ref{stint}) are 
\begin{eqnarray}
\mathbb{E}\left[ \int_{a}^{b}f(u)\,dX(u)\right] &=&\mathbb{E}\left[ X(1)%
\right] \int_{a}^{b}f(u)\,du\text{,}  \label{Xmean} \\
\mathrm{Var}\left[ \int_{a}^{b}f(u)\,dX(u)\right] &=&\mathrm{Var}\left[ X(1)%
\right] \int_{a}^{b}f^{2}(u)\,du\text{.}  \label{Xvar}
\end{eqnarray}
\end{lemma}

\begin{proof}
The stochastic integral (\ref{stint}), and the result (\ref{cumcum}) follow
from Lemma 2.5 of\newline
\citet{Barndorff-Nielsen2006}%
. By differentiating (\ref{cumcum}) and putting $z=0$ we obtain 
\begin{eqnarray*}
\mathbb{E}\left[ \int_{a}^{b}f(u)\,dX(u)\right]  &=&-i\int_{a}^{b}f(u)\dot{C}%
\left( 0;X(1)\right) \,du \\
&=&\mathbb{E}\left[ X(1)\right] \int_{a}^{b}f(u)\,du\text{,}
\end{eqnarray*}%
where we have used the fact that $-i\dot{C}(0;X(1))=\mathrm{E}\left[ X(1)%
\right] $. Differentiating (\ref{cumcum}) twice and putting $z=0$ we obtain%
\begin{eqnarray*}
\mathrm{Var}\left[ \int_{a}^{b}f(u)\,dX(u)\right]  &=&-\int_{a}^{b}f^{2}(u)%
\ddot{C}\left( 0;X(1)\right) \,du \\
&=&\mathrm{Var}\left[ X(1)\right] \int_{a}^{b}f^{2}(u)\,du\text{,}
\end{eqnarray*}%
where we have used the fact that $-\ddot{C}(0;X(1))=\mathrm{Var}\left[ X(1)%
\right] $. This shows the last two results. In particular, the two integrals
on the right-hand side of (\ref{Xmean}) and (\ref{Xvar}) are finite, each
being the integral of a continuous function on a compact interval.
\end{proof}

We shall now extend the integration interval to the whole real line. To this
end we first need to extend the time domain of the process $X$ to $\mathbb{R}
$. Since we are dealing with a L\'{e}vy process, we have that $X(0)\equiv 0$%
. We then extend the process to the negative half-axis by assuming 
\begin{equation}
-X(-t)=X(0)-X(-t)\overset{d}{=}X(t)\text{ for }t\geq 0\text{,}
\label{minust}
\end{equation}%
cf. \cite{Taqqu2003}.

\begin{proposition}
\label{Proposition}Assume that $-\infty \leq a<b\leq \infty $. Let $X(t)$ be
a L\'{e}vy process on $\mathbb{R}$, and let $f:(a,b)\rightarrow \mathbb{R}$
be continuous on $(a,b)\setminus \left\{ 0\right\} $. Assume that 
\begin{equation}
\forall z\in \mathbb{R}:\int_{a}^{b}\left\vert C\left( zf(u);X(1)\right)
\right\vert \,du<\infty \text{.}  \label{inteC}
\end{equation}%
Then the stochastic integral (\ref{stint}) exists in the limit, in
probability, of the sequence 
\begin{equation*}
\left\{ \int_{a_{n}}^{b_{n}}f(u)\,dX(u)\right\} _{n\in \mathbb{N}},
\end{equation*}%
where $a_{n}$ and $b_{n}$ are arbitrary sequences in $(a,b)$ such that $%
a_{n}\leq b_{n}$ for all $n$ and $a_{n}\downarrow a$ and $b_{n}\uparrow b$
as $n\rightarrow \infty $. The distribution of the random variable (\ref%
{stint}) is infinitely divisible, and has cumulant transform given by (\ref%
{cumcum}). If $X(t)$ has finite second moments, then the mean and variance
of the random variable (\ref{stint}) are given by (\ref{Xmean}) and (\ref%
{Xvar}), respectively.
\end{proposition}

\begin{proof}
First assume that $0\leq a<b\leq \infty $. In this case, the stochastic
integral and the result (\ref{cumcum}) follow from Proposition 2.6 of 
\citet{Barndorff-Nielsen2006}%
. In the case $-\infty \leq a<b\leq 0$ we make the substitution $v=-u$, and
rewrite the integral as follows: 
\begin{equation}
\int_{a}^{b}f(u)\,dX(u)=\int_{-a}^{-b}f(-v)\,dX(-v)=\int_{-b}^{-a}f(-v)\,d 
\left[ -X(-v)\right] \text{.}  \label{integrals}
\end{equation}%
In view of (\ref{minust}), the process $\left\{ -X(-t):t\geq 0\right\} $ is
a L\'{e}vy process on the positive half-line. Since $0\leq -b<-a\leq \infty $%
, the right-most integral of (\ref{integrals}) is defined according to the
first case just considered, thereby providing the required definition of the
left-most integral of (\ref{integrals}). In the case where $a$ is negative
and $b$ is positive, we split the integral in two parts, 
\begin{equation*}
\int_{a}^{b}f(u)\,dX(u)=\int_{a}^{0}f(u)\,dX(u)+\int_{0}^{b}f(u)\,dX(u)\text{%
,}
\end{equation*}%
and use the fact that both integrals on the right-hand side of the equation
are now properly defined. The remainder of the results follows from Lemma %
\ref{Lemma} by elementary arguments.
\end{proof}

\subsection{Fractional Hougaard motion\label{secFrac}}

We shall now define a class of fractional Hougaard motions by means of
stochastic integration with respect to a Hougaard L\'{e}vy process. This
approach is similar to the moving average representation of fractional
Brownian motion, where the integration is with respect to ordinary Brownian
motion \citep{Mandelbrot1968a}. Fractional Hougaard motions provide our main
examples of non-L\'{e}vy self-similar families of stochastic processes.

Following 
\citet[pp.~56--59]{Beran1994}
and 
\citet{Taqqu2003}%
, we define the weight function $w_{h}(t,u)$ for $t\geq 0$ and $u\in \mathbb{%
R}$ by 
\begin{equation}
w_{h}(t,u)=\left\{ 
\begin{array}{ccc}
(t-u)^{h}-(-u)^{h} & \text{ for } & u<0 \\ 
(t-u)^{h} & \text{ for } & 0\leq u<t \\ 
0 & \text{ for } & t\leq u\text{,}%
\end{array}%
\right.  \label{tu}
\end{equation}%
where $h$ is a real constant. Using the notation $u_{+}=\max \left\{
0,u\right\} $ we may also express the weight function in the more compact
form%
\begin{equation}
w_{h}(t,u)=(t-u)_{+}^{h}-(-u)_{+}^{h}\text{.}  \label{plus}
\end{equation}

For a given $p\in \Delta $, we let $\left\{ S_{p}(\mu ;t):t\in \mathbb{R}%
\right\} $ denote the Hougaard L\'{e}vy process defined in Section~\ref%
{ExHou}, extended to the whole real line by means of (\ref{minust}). For
simplicity we take $\sigma ^{2}=1,$ and we recall from (\ref{alpha}) that $%
\alpha =1+(1-p)^{-1}$.

\begin{proposition}
\label{fracHou}Assume that $p\in \Delta \setminus \left\{ 1,2\right\} $ and $%
H<1/\alpha $, and define the \emph{fractional Hougaard motion} $S_{p,H}(\mu
;t)$ by the stochastic integral 
\begin{equation}
S_{p,H}(\mu ;t)=\int_{-\infty }^{\infty }w_{h}(t,u)\,dS_{p}(\mu ;u)\text{
for }t\geq 0\text{,}  \label{yphh}
\end{equation}%
where $h=H-1/\alpha $ and $\mu \in \overline{\Omega }_{p}$. Then the family $%
S_{p,H}(\mu ;t)$ is $H$-SSSI with cumulant transform 
\begin{equation}
C(z;S_{p,H}(\mu ;t))=\frac{\alpha -1}{\alpha }\mu ^{\alpha /(\alpha
-1)}\int_{-\infty }^{\infty }\left\{ \left[ 1+\frac{izw_{h}(t,u)}{(\alpha
-1)\mu ^{1/(\alpha -1)}}\right] ^{\alpha }-1\right\} \,du\text{ for }z\in 
\mathbb{R}\text{.}  \label{ctransform}
\end{equation}%
For $1/\alpha -1<H<1/\alpha $, the process $S_{p,H}(\mu ;t)$ has mean zero,
and for $1/\alpha -1/2<H<1/\alpha $, the process has variance%
\begin{equation}
\mathrm{Var}\left[ S_{p,H}(\mu ;t)\right] =\mu ^{p}t^{1+2\left( H-1/\alpha
\right) }\int_{-\infty }^{\infty }w_{h}^{2}(1,v)\,dv\text{.}  \label{VarTw}
\end{equation}
\end{proposition}

\begin{proof}
Our method of proof is inspired by 
\citet[pp.~56--59]{Beran1994}%
, and we shall hence show self-similarity by a direct argument, rather than
via the covariance function, say, as is usually done for fractional Brownian
motion. We first need to check that the integral (\ref{inteC}) is finite for
the function $f(u)=w_{h}(t,u)$. Let us define the complex analytic function $%
\kappa _{\alpha }$ by 
\begin{equation}
\kappa _{\alpha }(x)=\frac{\alpha -1}{\alpha }\left( \frac{x}{\alpha -1}%
\right) ^{\alpha }\text{ for }\func{Re}\frac{x}{\alpha -1}>0  \label{kalpha}
\end{equation}%
\citep[p.~131]{Jorgensen1997}%
, where the domain may be extended to the imaginary axis when $\alpha >0$,
and to the whole complex plane when $\alpha =2$. Then $S_{p}(\mu ;1)$ has
cumulant transform%
\begin{equation*}
C\left( z;S_{p}(\mu ;1)\right) =\kappa _{\alpha }(\theta +iz)-\kappa
_{\alpha }(\theta )=\kappa _{\alpha }(\theta )\left[ \left( 1+\frac{iz}{%
\theta }\right) ^{\alpha }-1\right] \text{,}
\end{equation*}%
where $\theta =(\alpha -1)\mu ^{1/(\alpha -1)}$. The integral (\ref{inteC})
hence takes the form 
\begin{equation*}
\int_{-\infty }^{\infty }\left\vert C\left( zw_{h}(t,u);S_{p}(\mu ;1)\right)
\right\vert \,du=\int_{-\infty }^{\infty }\left\vert \kappa _{\alpha
}(\theta +izw_{h}(t,u))-\kappa _{\alpha }(\theta )\right\vert \,du.
\end{equation*}%
In the limit $u\rightarrow -\infty $, a Taylor expansion of $\kappa _{\alpha
}$ yields 
\begin{equation*}
\left\vert \kappa _{\alpha }(\theta +izw_{h}(t,u))-\kappa _{\alpha }(\theta
)\right\vert =O\left( \left\vert w_{h}(t,u)\right\vert \right) \text{,}
\end{equation*}%
which implies integrability in this limit, due the assumption that $%
h=H-1/\alpha <0$. For $u$ near zero or $t$, the integrability of $\left\vert
\kappa _{\alpha }(\theta +izw_{h}(t,u))-\kappa _{\alpha }(\theta
)\right\vert =O\left( \left\vert w_{h}(t,u)\right\vert ^{\alpha }\right) $
requires $\alpha h=\alpha H-1>-1$, or equivalently $H\alpha >0$, which is
satisfied because $\alpha \leq 2$ and $H<1/\alpha $. By (\ref{cumcum}) we
find that the process $S_{p,H}(t)$ has cumulant transform 
\begin{eqnarray*}
C(z;S_{p,H}(\mu ;t)) &=&\int_{-\infty }^{\infty }C\left(
zw_{h}(t,u);S_{p}(\mu ;1)\right) \,du \\
&=&\int_{-\infty }^{\infty }\left[ \kappa _{\alpha }(\theta
+izw_{h}(t,u))-\kappa _{\alpha }(\theta )\right] \,du \\
&=&\kappa _{\alpha }(\theta )\int_{-\infty }^{\infty }\left\{ \left[ 1+\frac{%
iz}{\theta }w_{h}(t,u)\right] ^{\alpha }-1\right\} \,du\text{,}
\end{eqnarray*}%
which implies (\ref{ctransform}), in view of the above definition of $\theta 
$.

Now let us calculate the mean and variance of $S_{p,H}(\mu ;t)$ using (\ref%
{Xmean}) and (\ref{Xvar}). Using well-known results about the integral of $%
w_{h}$ 
\citep[e.g.][formula (2.22) p.~59]{Beran1994}
we obtain that the mean is finite for $1/\alpha -1<H<1/\alpha $ and 
\begin{equation*}
\,\mathbb{E}\left[ S_{p,H}(\mu ;t)\right] =\mathbb{E}\left[ S_{p}(\mu ;1)%
\right] \int_{-\infty }^{\infty }w_{h}(t,u)\,du=0\text{.}
\end{equation*}%
To calculate the variance we use the following scaling relation for the
weight function $w_{h}$,%
\begin{equation}
w_{h}\left( t,u\right) =t^{h}w_{h}(1,ut^{-1})\text{.}  \label{powert}
\end{equation}%
Using the substitution $v=ut^{-1}$, and with $h=H-1/\alpha $, we obtain for $%
1/\alpha -1/2<H<1/\alpha $,%
\begin{eqnarray*}
\mathrm{Var}\left[ S_{p,H}(\mu ;t)\right] &=&\mu ^{p}\int_{-\infty }^{\infty
}w_{h}^{2}(t,u)\,du \\
&=&\mu ^{p}t^{2h}\int_{-\infty }^{\infty }w_{h}^{2}(1,ut^{-1})\,du \\
&=&\mu ^{p}t^{1+2(H-1/\alpha )}\int_{-\infty }^{\infty }w_{h}^{2}(1,v)\,dv%
\text{.}
\end{eqnarray*}

In order to show that the family of fractional Hougaard motions $S_{p,H}(\mu
;t)$ is $H$-SS, we shall use the fact that the Hougaard L\'{e}vy process $%
S_{p}(\mu ;t)$ is $1/\alpha $-SS with $\alpha =1+(1-p)^{-1}$. Consider (\ref%
{yphh}) with arguments $\mu c^{H-1}$ and $ct$. We note that the function $%
w_{h}$ satisfies the following scaling relation similar to (\ref{powert}), 
\begin{equation}
w_{h}(ct,u)=c^{h}w_{h}\left( t,uc^{-1}\right) \text{ for }c>0\text{.}
\label{scaling1}
\end{equation}%
In view of (\ref{scaling1}), we obtain 
\begin{eqnarray*}
S_{p,H}(\mu c^{H-1};ct) &=&\int_{-\infty }^{\infty }w_{h}(ct,u)\,dS_{p}(\mu
c^{H-1};u) \\
&=&c^{h}\int_{-\infty }^{\infty }w_{h}(t,uc^{-1})\,dS_{p}(\mu c^{H-1};u)%
\text{.}
\end{eqnarray*}%
Making the substitution $v=uc^{-1}$, this becomes 
\begin{equation*}
S_{p,H}(\mu c^{H-1};ct)=c^{h}\int_{-\infty }^{\infty }w_{h}(t,v)\,dS_{p}(\mu
c^{H-1};cv)\text{.}
\end{equation*}%
Recalling that $h=H-1/\alpha $, and using the $1/\alpha $-SS property of $%
S_{p}(\mu ;t)$, this implies 
\begin{equation*}
S_{p,H}(\mu c^{H-1};ct)\overset{d}{=}c^{H-1/\alpha }\int_{-\infty }^{\infty
}w_{h}(t,v)c^{1/\alpha }dS_{p}(\mu ;v)=c^{H}S_{p,H}(\mu ;t)\text{.}
\end{equation*}%
Hence, $S_{p,H}(\mu ;t)$ is self-similar with Hurst exponent $H$.

To show that $S_{p,H}(\mu ;t)$ has stationary increments, we observe, using
the expression (\ref{plus}) for $w_{h}$, that for any $s,t>0$,%
\begin{align*}
S_{p,H}(\mu ;t+s)-S_{p,H}(\mu ;t)& =\int_{-\infty }^{\infty }\left\{
w_{h}(t+s,u)-w_{h}(t,u)\right\} \,dS_{p}(\mu ;u) \\
& =\int_{-\infty }^{\infty }\left\{ (t+s-u)_{+}^{h}-(t-u)_{+}^{h}\right\}
\,dS_{p}(\mu ;u) \\
& \overset{d}{=}\int_{-\infty }^{\infty }\left\{
(s-v)_{+}^{h}-(-v)_{+}^{h}\right\} \,dS_{p}(\mu ;v) \\
& =S_{p,H}(\mu ;s)\text{,}
\end{align*}%
where, in the second-last equality, we have used the substitution $v=u-t$.
Since $S_{p,H}(\mu ;0)\equiv 0$, this is the desired result, completing the
proof.
\end{proof}

As already mentioned, fractional Hougaard motion is our main example of
general self-similarity outside of the L\'{e}vy process case. However, the
processes $S_{p,H}(\mu ;s)$ do not have the same marginal distribution as
the Hougaard L\'{e}vy processes of Section~\ref{ExHou}. In particular, for $%
p>1,$ $S_{p,H}(\mu ;s)$ has support on the whole real line rather than on $%
[0,\infty )$ or $(0,\infty )$. This follows from the fact that the parameter 
$h$ is restricted to negative values, which implies that $w_{h}(t,u)$ is
negative for $u<0$ and positive for $0<u<t$, so that $S_{p,H}(\mu ;s)$ has
support on the whole real line.

\section{Lamperti transformations\label{secLampertiTrans}}

We shall now introduce an extension of the \emph{Lamperti transformation }to
the case of general self-similarity. Consider a strictly $H$-SS stochastic
process $X(t)$, and recall that the Lamperti transformation yields a new
strictly stationary stochastic process $\left\{ Y(t):t\in \mathbb{R}\right\} 
$ by means of the following exponential scaling transformation, 
\begin{equation}
Y(t)=e^{-tH}X(e^{t})\text{ for }t\in \mathbb{R}\text{,}  \label{Ordinary}
\end{equation}%
see \cite{Lamperti1962} and 
\citet[p.~11]{Embrechts2002a}%
. Now, let $X$ be a general $H$-SS family, and define the family of
processes $Y(\mu ;t)$ for $\mu \in \overline{\Omega }$ by 
\begin{equation}
Y(\mu ;t)=e^{-tH}X(\mu e^{t(H-1)};e^{t})\text{ for }t\in \mathbb{R}\text{.}
\label{Lamperti}
\end{equation}%
Equation (\ref{Lamperti}) may be obtained from (\ref{fixedpoint}) by taking $%
t=1$ and then substituting $c=e^{t}$. The fact that the right-hand side of (%
\ref{fixedpoint}) does not depend on $c$ suggests that $Y(\mu ;t)$ might be
stationary. By the self-similarity of $X$ we obtain, for each $\mu \in 
\overline{\Omega }$,%
\begin{equation}
Y(\mu ;t)=e^{-tH}X(\mu e^{t(H-1)};e^{t})\overset{d}{\sim }X(e^{-t(H-1)}\mu
e^{t(H-1)};e^{-t}e^{t})=X(\mu ;1)\text{ for }t\in \mathbb{R}\text{.}
\label{OU-Lamperti}
\end{equation}%
This result shows that the marginals of $Y(\mu ;t)$ are stationary as a
function of $t$ for each $\mu $.

Note that in the special case of a process with drift, $X(\mu ;t)=X(t)+\mu t$%
, say, the process (\ref{Lamperti}) takes the form 
\begin{equation}
Y(\mu ;t)=e^{-tH}X(e^{t})+\mu =Y(t)+\mu \text{ for }t\in \mathbb{R}\text{,}
\label{stationary}
\end{equation}%
which is simply the process (\ref{Ordinary}) with a shift $\mu \in \mathbb{R}
$. If we assume that the drift term $\mu t$ is added pathwise to $X(t),$
then the shift $\mu $ is similarly added pathwise to $Y(t)$. In this case, (%
\ref{Lamperti}) is equivalent to the conventional Lamperti transformation.
In the special case where $X$ is a Brownian motion with drift, the
corresponding family $Y$ consists of shifted Ornstein-Uhlenbeck processes,
and (\ref{stationary}) represents a class of generalized Ornstein-Uhlenbeck
processes.

Let us now consider the random field $\left\{ X(\mu ;t):t\geq 0,\mu \in 
\overline{\Omega }\right\} $, and assume that (\ref{fixedpoint}) holds in
the sense of equality of the finite-dimensional distributions of the two
random fields. In this case, the stationarity of the process $Y(\mu ;\cdot )$
may be shown as follows:\ 
\begin{align*}
Y(\mu ;s+t)& =e^{-\left( s+t\right) H}X(\mu e^{\left( s+t\right)
(H-1)};e^{s+t}) \\
& =e^{-tH}e^{-sH}X(\mu e^{s(H-1)}e^{t(H-1)};e^{s}e^{t}) \\
& \overset{d}{=}e^{-tH}X(\mu e^{t(H-1)};e^{t}) \\
& =Y(\mu ;t)\text{ for }t\in \mathbb{R}\text{,}
\end{align*}%
for all $\mu \in \overline{\Omega }$. The stationary processes $Y(\mu ;\cdot
)$ are generalized Ornstein-Uhlenbeck processes.

Let us now turn to the inverse Lamperti transformation. Suppose that $%
\left\{ Y(\mu ;t):\mu \in \overline{\Omega },\;t\in \mathbb{R}\right\} $ is
a family of strictly stationary processes. For a given $H\in \mathbb{R}$, we
define the family $X$ by 
\begin{equation}
X(\mu ;t)=t^{H}Y(\mu t^{1-H};\log t)\text{ for }t>0  \label{invLam}
\end{equation}%
for all $\mu \in \overline{\Omega }$. In the special case where $Y(\mu
;t)=Y(t)+\mu $, namely a stationary process $Y(t)$ plus a shift $\mu \in 
\mathbb{R}$, then $X$ has the form $X(\mu ;t)=t^{H}Y(\log t)+\mu t$, which
is a strictly $H$-SS process plus a drift term.

The discussion of the self-similarity of $X$ is similar to be above
discussion for the ordinary Lamperti transformation (\ref{Lamperti}). Let us
assume that $\left\{ Y(\mu ;t):\mu \in \overline{\Omega },\;t\in \mathbb{R}%
\right\} $ is a $t$-stationary random field, in the sense that 
\begin{equation*}
Y(\mu ;t+s)\overset{d}{=}Y(\mu ;t)\text{ for }\mu \in \overline{\Omega }%
\text{ and}\;t\in \mathbb{R}\text{,}
\end{equation*}%
for all $s\in \mathbb{R}$, where $\overset{d}{=}$ here means equality of the
finite-dimensional distributions of the two random fields. Under this
assumption we obtain, for $X$ defined by (\ref{invLam}), 
\begin{eqnarray*}
c^{H}X(\mu ;t) &=&\left( ct\right) ^{H}Y(\mu t^{1-H};\log t) \\
&=&\left( ct\right) ^{H}Y(\mu c^{H-1}\left( ct\right) ^{1-H};\log t) \\
&\overset{d}{=}&\left( ct\right) ^{H}Y(\mu c^{H-1}\left( ct\right)
^{1-H};\log c+\log t) \\
&=&X(\mu c^{H-1};ct)\text{,}
\end{eqnarray*}%
which shows that the family $X$ is $H$-SS.

\section{Lamperti-type limit theorems\label{secLamperti}}

We shall now consider some generalizations of Lamperti's limit theorem of 
\citep{Lamperti1962}%
, according to which all self-similar stochastic processes appear as
large-sample limits of suitably scaled stochastic processes.

Lamperti's fundamental limit theorem says that if we are given a stochastic
process $X(t)$ and a positive measurable function $A(c)$ satisfying 
\begin{equation}
A(c)\rightarrow \infty \text{ as }c\rightarrow \infty ,  \label{biga}
\end{equation}%
and%
\begin{equation}
A(c)^{-1}X(ct)\overset{d}{\rightarrow }Y(t)\text{ for }t\geq 0\text{,}
\label{ytox}
\end{equation}%
as $c\rightarrow \infty ,$ where the stochastic process $Y(t)$ has
non-degenerate marginal distributions for all $t>0$, then $Y(t)$ is $H$-SS
for some $H>0$, and all $H$-SS processes appear as limits of this form.
Moreover, the function $A$ is regularly varying with index $H$, that is, 
\begin{equation}
A(c)=c^{H}L(c)\text{,}  \label{slow}
\end{equation}%
where $L$ is a slowly varying function. See e.g. %
\citet[p.~14]{Embrechts2002a} for a proof.

Let us first discuss an infinitely divisible asymptotic version of
Lamperti's limit theorem. Hence, let us replace the condition (\ref{biga})
by 
\begin{equation*}
A(c)\rightarrow \infty \text{ as }c\downarrow 0,
\end{equation*}%
and let us assume that (\ref{ytox}) now holds for $c\downarrow 0$. Then a
straightforward modification of Lamperti's proof shows that $Y(t)$ is again $%
H$-SS for some $H>0$, and (\ref{slow}) holds with $L$ slowly varying at zero.

Let us consider the processes with drift corresponding to $X(t)$ and $Y(t)$,
namely $X(\mu ;t)=X(t)+\mu t$ and $Y(\mu ;t)=Y(t)+\mu t$, say. Then (\ref%
{ytox}) implies that%
\begin{align*}
A(c)^{-1}X(\mu A(c)c^{-1};ct)& =A(c)^{-1}X(ct)+\mu t \\
& \overset{d}{\rightarrow }Y(t)+\mu t=Y(\mu ;t)\text{ for }t\geq 0\text{,}
\end{align*}%
as $c\rightarrow \infty $ or $c\downarrow 0$, where $Y(\mu ;t)$ is a general
self-similar family of stochastic processes.

These results motivate the following question. Suppose we are given families
of stochastic processes $X(\mu ;t)$ and $Y(\mu ;t)$, and a function $A(c)$
such that 
\begin{equation}
A(c)^{-1}X(\mu A(c)c^{-1};ct)\overset{d}{\rightarrow }Y(\mu ;t)\text{ for }%
t\geq 0\text{, as }c\rightarrow \infty \text{ or }c\downarrow 0\text{.}
\label{LampertiConvergence}
\end{equation}%
Under which conditions does this imply that the family of stochastic
processes $Y$\ is self-similar? We shall not answer this question in full
generality here, but clearly (\ref{LampertiConvergence}) includes as special
cases the different cases of convergence discussed so far. The following
partial result shows a case where (\ref{LampertiConvergence}) applies to
non-L\'{e}vy processes.

Let us replace $A(c)$ by $c^{H}$ in (\ref{LampertiConvergence}), which
yields the condition 
\begin{equation}
c^{-H}X(\mu c^{H-1};ct)\overset{d}{\rightarrow }Y(\mu ;t)\text{ for }t\geq 0%
\text{, as }c\rightarrow \infty \text{ or }c\downarrow 0\text{,}
\label{powerc}
\end{equation}%
which corresponds to convergence of the action of the renormalization group (%
\ref{group}) to the fixed point (\ref{fixedpoint}). We then have the
following Lamperti-type limit theorem, based on an adaptation of the
original proof, as given by \citet[p.~14]{Embrechts2002a}.

\begin{theorem}
Consider a family of stochastic processes $Y(\mu ;t)$.

\begin{enumerate}
\item If there exists a family of stochastic processes $X(\mu ;t)$ such that
(\ref{powerc}) holds, then the process $Y(\mu ;t)$ is $H$-SS for each $\mu
\in \overline{\Omega }$.

\item If $Y(\mu ;t)$ is $H$-SS then there exists a family $X(\mu ;t)$
satisfying (\ref{powerc}).
\end{enumerate}
\end{theorem}

\begin{proof}
We firs show Item 1. in the case $c\rightarrow \infty .$ The proof in the
case $c\downarrow 0$ is similar. By (\ref{powerc}), for any $t_{1},\ldots
,t_{k}>0$ and for continuity points $x_{1},\ldots ,x_{k}$ of $\left\{ Y(\mu
;t_{j}),1\leq j\leq k\right\} $ and $\left\{ Y(\mu b^{H-1};bt_{j}),1\leq
j\leq k\right\} $, where $b>0$, we obtain 
\begin{equation}
\lim_{c\rightarrow \infty }\Pr \left\{ c^{-H}X(\mu c^{H-1};ct_{j})\leq
x_{j},1\leq j\leq k\right\} =\Pr \left\{ Y(\mu ;t_{j})\leq x_{j},1\leq j\leq
k\right\} \text{.}  \label{aa}
\end{equation}%
Replacing $t_{j}$ by $bt_{j}$, $\mu $ by $\mu b^{H-1}$ and $x_{j}$ by $%
x_{j}b^{H}$ we obtain 
\begin{equation}
\lim_{c\rightarrow \infty }\Pr \left[ c^{-H}X\left\{ \mu \left( cb\right)
^{H-1};\left( cb\right) t_{j}\right\} \leq x_{j}b^{H},1\leq j\leq k\right]
=\Pr \left\{ Y(\mu b^{H-1};bt_{j})\leq x_{j}b^{H},1\leq j\leq k\right\} 
\text{.}  \label{bb}
\end{equation}%
Since the limits on the left-hand sides of (\ref{aa}) and (\ref{bb}) have
the same value, due to (\ref{powerc}), it follows by comparing the two
right-hand sides that $Y(\mu ;t)\overset{d}{=}b^{-H}Y(\mu b^{H-1};bt)$,
which is the self-similarity of $Y$.

Item 2. is trivially obtained by taking $X(\mu ;t)\equiv Y(\mu ;t)$. This
completes the proof.
\end{proof}

\section{Exponential variance functions ($H=1$)\label{twinfty}}

Much like in the case of strict self-similarity, the value $H=1$ for the
Hurst exponent corresponds to a degenerate family of processes, at least
when second moments are finite. In fact, if the family $X$ is SSSI in the
sense of Definition \ref{ss-def}, the scaling (\ref{H-ss}) with $H=1$
becomes 
\begin{equation}
X(\mu ;ct)\overset{d}{=}cX(\mu ;t)\text{ for }t\geq 0\text{.}  \label{H=1}
\end{equation}%
Under second-moment assumptions, this implies that $\mathrm{Cov}\left[ X(\mu
;s),X(\mu ;t)\right] =V(\mu )st$ (compare with (\ref{gamma-case})). For $t=1$
it follows, in turn, that 
\begin{eqnarray}
\mathrm{Var}\left[ X(\mu ;s)-sX(\mu ;1)\right] &=&\mathrm{Var}\left[ X(\mu
;s)\right] +s^{2}\mathrm{Var}\left[ X(\mu ;1)\right] -2s\mathrm{Cov}\left[
X(\mu ;s),X(\mu ;1)\right]  \notag \\
&=&V(\mu )s^{2}+V(\mu )s^{2}-2sV(\mu )s=0\text{. }  \label{deg}
\end{eqnarray}%
This implies that $X(\mu ;\cdot )$ is a degenerate, straight-line processes $%
X(\mu ;s)\equiv sX(\mu ;1)$ a.s. for each $\mu \in \Omega $.

Let us instead propose an extended definition of self-similarity for the
case $H=1$. We say that a family of stochastic processes $X$ is \emph{%
self-similar with Hurst exponent} $H=1$ ($1$-SSSI) if there exists a
function $f:\mathbb{R}_{+}\rightarrow \overline{\Omega }$ such that for all $%
c>0$ and $\mu \in \overline{\Omega }$. 
\begin{equation}
X(\mu +f(c);ct)\overset{d}{=}c\left[ X(\mu ;t)+tf(c)\right] \text{ for }t>0%
\text{.}  \label{pis1}
\end{equation}%
The interpretation of this definition is that a location change for the rate
and a simultaneous rescaling of time is equivalent to a location and scale
change for the process. In the special case $f(c)\equiv 0$ we obtain the
degenerate case (\ref{H=1}) already discussed.

\begin{proposition}
\label{1mean}Let $X$ be a $1$-SSSI family of stochastic processes satisfying
(\ref{pis1}) with a non-constant $f$. Assume that $X(\mu ;t)$ has finite
expectation for all $\mu \in \Omega $ and $t\geq 0$. If the function $%
t\mapsto \mathbb{E}\left[ X(\mu ;t)\right] $ is continuous on $[0,\infty )$,
then up to a translation of $\mu $ there exist constants $a$ and $b\neq 0$
such that%
\begin{equation}
\mathbb{E}\left[ X(\mu ;t)\right] =ae^{b\mu }+\mu t\text{ for }t\geq 0\text{
and }\mu \in \Omega \text{.}  \label{meanfunction}
\end{equation}%
When $a\neq 0$, this implies that 
\begin{equation}
f(c)=b^{-1}\log c\text{ for }c>0\text{.}  \label{fc}
\end{equation}
\end{proposition}

\begin{proof}
By the proof of Proposition \ref{propmean}, the stationarity of the
increments for $X$ implies that 
\begin{equation*}
\mathbb{E}\left[ X(\mu ;t)\right] =a(\mu )+b(\mu )t\text{,}
\end{equation*}%
for suitable functions $a$ and $b$. In view of (\ref{pis1}), this implies
that the functions $a$, $b$ and $f$ satisfy%
\begin{equation*}
a(\mu +f(c))+b(\mu +f(c))ct=ca(\mu )+\left[ b(\mu )+f(c)\right] ct\text{ }
\end{equation*}%
for $c,t>0$. By comparing the left- and right-hand linear functions of $t$
we obtain the equations%
\begin{eqnarray}
a(\mu +f(c)) &=&ca(\mu )  \label{aanda} \\
b(\mu +f(c)) &=&b(\mu )+f(c)\text{.}  \label{bandb}
\end{eqnarray}%
Since $f$ is not constant, the solution to (\ref{aanda}) is $a(\mu
)=ae^{b\mu }$, say, where $a=a(0)$ and $b\neq 0$ are constants. When $a\neq
0 $, this implies $f(c)=b^{-1}\log c$. From (\ref{bandb}) we obtain $b(\mu
)=b(0)+\mu $, where we may take $b(0)=0$ up to a translation of $\mu $.
\end{proof}

When the family $X$ satisfies the conditions of Proposition \ref{1mean}, we
define the centered family $X_{0}$ by analogy with (\ref{famcen}), that is, 
\begin{equation*}
X_{0}(\mu ;t)=X(\mu ;t)-ae^{b\mu }\text{.}
\end{equation*}%
It is easy to check that the family $X_{0}$ is again self-similar, now with $%
a=0$. The mean function (\ref{meanfunction}) for $X_{0}$ then takes the form 
\begin{equation}
\mathbb{E}\left[ X_{0}(\mu ;t)\right] =\mu t\text{ for }t\geq 0\text{ and }%
\mu \in \Omega \text{.}  \label{form}
\end{equation}
Unless otherwise stated, we always assume that the $1$-SSSI family $X$ is
centered.

Let us now derive the covariance function of an $1$-SSSI family $X$,
assuming that the rate $\mu $ and variance function $V$ satisfy (\ref{form})
and (\ref{unitvf}), respectively. A simple rearrangement of (\ref{pis1})
with $c=t^{-1}$ yields 
\begin{equation*}
X(\mu ;t)\overset{d}{\sim }tX(\mu +f\left( t^{-1}\right) );1)-tf\left(
t^{-1}\right) \text{ for }t>0\text{,}
\end{equation*}%
which gives the following expression for the variance of the process, 
\begin{equation}
\mathrm{Var}\left[ X(\mu ;t)\right] =t^{2}V(\mu +f\left( t^{-1}\right)
)=V_{1}(\mu ;t)\text{,}  \label{var1fu}
\end{equation}%
say. The covariance function now looks similar to the general case, namely 
\begin{equation}
\mathrm{Cov}\left[ X(\mu ;s),X(\mu ;t)\right] =\frac{1}{2}\left[ V_{1}(\mu
;s)+V_{1}(\mu ;t)-V_{1}(\mu ;\left\vert t-s\right\vert )\right] \text{,}
\label{covexp}
\end{equation}%
for $s,t>0$.

Following 
\citet[Ch.~4]{Jorgensen1997}%
, we consider the family of infinitely divisible exponential dispersion
models $\mathrm{Tw}_{\infty }(\mu ,b,t)$, indexed by the parameter $b\in 
\mathbb{R}$ and with domain $\Omega =\mathbb{R}$ for the rate parameter $\mu 
$. For given $b$, the model $\mathrm{Tw}_{\infty }(\mu ,b,t)$ is defined by
the variance function $V(\mu )=\sigma ^{2}e^{b\mu }$ for $\mu \in \mathbb{R}$
and $\sigma ^{2}>0$. The case $b=0$ corresponds to the normal distribution
with constant variance function, but for $b\neq 0$, we may consider $\mathrm{%
Tw}_{\infty }(\mu ,b,t)$ to be a Tweedie model with power parameter $%
p=\infty $. For $b>0$ ($b<0$), the model \textrm{$Tw$}$_{\infty }(\infty
,b,t)$ (\textrm{$Tw$}$_{\infty }(-\infty ,b,t)$) is an extreme $1$-stable
distribution, and we hence obtain the extended domain $\overline{\Omega }%
_{b}=(-\infty ,\infty ]$ for $b>0$ and $\overline{\Omega }_{b}=[-\infty
,\infty )$ for $b<0$.

The model $\mathrm{Tw}_{\infty }(\mu ,b,t)$ satisfies two separate
transformation properties 
\cite[Ch.~4]{Jorgensen1997}%
, where the first is a translation property,%
\begin{equation}
c+\mathrm{Tw}_{\infty }(\mu ,b,t)=\mathrm{Tw}_{\infty }(c+\mu ,b,te^{bc})%
\text{ for }c\in \mathbb{R}\text{,}  \label{translate}
\end{equation}%
and the second is a scaling property, 
\begin{equation}
c\mathrm{Tw}_{\infty }(\mu ,b,t)=\mathrm{Tw}_{\infty }(c\mu ,bc^{-1},tc^{-2})%
\text{ for }c>0\text{.}  \label{scale}
\end{equation}%
The latter leads us to define $S_{\infty }(\mu ,b;t)$ as the L\'{e}vy
process with marginal distribution 
\begin{eqnarray}
S_{\infty }(\mu ,b;t) &\overset{d}{\sim }&t\mathrm{Tw}_{\infty }(\mu ,b,t) 
\notag \\
&=&\mathrm{Tw}_{\infty }(\mu t,bt^{-1},t^{-1})\text{ for }t>0\text{.}
\label{marginal}
\end{eqnarray}%
Depending on the sign of $b$, we find that $S_{\infty }(\infty ,b;t)$ ($b>0$%
), respectively $S_{\infty }(-\infty ,b;t)$ ($b<0$), are extreme $1$-stable L%
\'{e}vy processes. The exponential family of stochastic processes $S_{\infty
}(\mu ,b;t)$ for $\mu \in \overline{\Omega }_{b}$ may hence be generated
from the respective processes $S_{\infty }(\pm \infty ,b;t)$ by exponential
tilting. As is the case for ordinary Hougaard L\'{e}vy processes, the
families $S_{\infty }(\mu ,b;t)$ may be characterized by self-similarity, as
we shall now see.

\begin{theorem}
\label{extremeHougaard} Let $X$ be a non-degenerate exponential family of L%
\'{e}vy processes. Then $X$ is self-similar with Hurst exponent $H=1$ in the
sense (\ref{pis1}) if and only if $X$ is a Hougaard family of L\'{e}vy
processes $S_{\infty }(\mu ,b;t)$ with $b\neq 0$.
\end{theorem}

\begin{proof}
Let us first show that the family of process $S_{\infty }(\mu ,b;t)$ for $%
\mu \in \overline{\Omega }_{b}$ is $1$-SSSI and satisfies the relation%
\begin{equation}
S_{\infty }(\mu +b^{-1}\log c,b;ct)\overset{d}{=}c\left\{ S_{\infty }(\mu
,b;t)+tb^{-1}\log c\right\} \text{ for }t>0\text{,}  \label{infinity}
\end{equation}%
corresponding to (\ref{pis1}) and (\ref{fc}) (with $f(c)=b^{-1}\log c$).
Since we have a L\'{e}vy process, it is enough to consider the marginal
distribution of $S_{\infty }(\mu ,b;t)$. Using (\ref{marginal}) and (\ref%
{translate}), we find that the left-hand side of (\ref{infinity}) has
marginal distribution%
\begin{equation*}
S_{\infty }(\mu +b^{-1}\log c,b;ct)\overset{d}{\sim }\mathrm{Tw}_{\infty
}(ct(\mu +b^{-1}\log c),b\left( ct\right) ^{-1},\left( ct\right) ^{-1})\text{%
.}
\end{equation*}%
The right-hand side of (\ref{infinity}) follows from (\ref{marginal}), (\ref%
{translate}) and (\ref{scale}), 
\begin{eqnarray*}
c\left\{ S_{\infty }(\mu ,b;t)+tb^{-1}\log c\right\} &\overset{d}{\sim }&c%
\mathrm{Tw}_{\infty }(\mu t+tb^{-1}\log c,bt^{-1},t^{-1}c) \\
&=&\mathrm{Tw}_{\infty }(ct\left( \mu +b^{-1}\log c\right) ,b\left(
ct\right) ^{-1},\left( ct\right) ^{-1})
\end{eqnarray*}%
This implies the $1$-SS property (\ref{infinity}), as desired.

To show that self-similarity with $H=1$ characterizes the Hougaard L\'{e}vy
process, let us consider an exponential family of L\'{e}vy processes $X$
with variance function $V$, such that (\ref{pis1}) is satisfied. We know
that $V$ is positive and analytic on $\Omega $, due to the exponential
family assumption, so $f$ is not identically zero. Since we are dealing with
a L\'{e}vy process, the variance of the process is $\mathrm{Var}X(\mu
;t)=tV(\mu )$. Comparing this with (\ref{var1fu}) we obtain the equation $%
tV(\mu )=t^{2}V(\mu +f(t^{-1}))$, or equivalently, with $s=t^{-1}$,%
\begin{equation}
V(\mu +f(s))=sV(\mu )\text{.}  \label{funceq}
\end{equation}%
If $f(c)\equiv 0$, we obtain the degenerate process (\ref{H=1}), which is
ruled out by assumption. Hence, by the same type of arguments as used in
connection with (\ref{aanda}), the solution to (\ref{funceq}) is 
\begin{equation}
V(\mu )=V(0)e^{b\mu }\text{ for }\mu \in \mathbb{R}  \label{Varbeta}
\end{equation}%
for some $b\neq 0$, and consequently $f(s)=b^{-1}\log s$. Taking $\sigma
^{2}=V(0),$ the variance function (\ref{Varbeta}) characterizes the family $%
\mathrm{Tw}_{\infty }(\mu ,b,t)$, and hence the family of Hougaard L\'{e}vy
processes $S_{\infty }(\mu ,b;t)$, completing the proof.
\end{proof}

\section*{Acknowledgements}

We are grateful to Florin Avram, Martin Jacobsen, Wayne Kendal, Gennady
Samorodnitsky, Steen Thorbj\o rnsen and Vladimir Vinogradov for helpful
comments on the paper. This research was supported by the Danish Natural
Science Research Council and FAPESP.

\bibliographystyle{abbrvnat}
\bibliography{bentj}

%
\newpage

\end{document}